\documentclass[12pt]{amsart}

\usepackage{times,txfonts}
\usepackage{amsmath, amssymb}
\usepackage{amsrefs}
\usepackage{hyperref}
\usepackage{setspace}
\numberwithin{equation}{section}

\newtheorem{thm}{Theorem}[section]
\newtheorem{prp}[thm]{Proposition}
\newtheorem{cor}[thm]{Corollary}

\newcommand{\e}{\varepsilon}
\newcommand{\C}{\mathbb{C}}
\newcommand{\R}{\mathbb{R}}
\newcommand{\N}{\mathbb{N}}
\newcommand{\tr}{{\rm Tr}}

\newcommand{\rank}{{\rm rank}}
\newcommand{\HS}{{\rm HS}}
\newcommand{\srank}{{\rm srank}}
\newcommand{\E}{\mathbb{E}}
\newcommand{\p}{\mathbb{P}}
\newcommand{\supp}{\rm supp}
\newcommand{\aprox}{{\bf Approx}}
\newcommand{\eqdef}{\stackrel{\mathrm{def}}{=}}

\begin{document}

\title{Approximating matrices and convex bodies through Kadison-Singer}

\author{Omer Friedland}
\address{Institut de Math\'ematiques de Jussieu\\ Universit\'e Pierre et Marie Curie (Paris 6)\\ 4 place Jussieu, 75252 Paris, France}
\email{\tt omer.friedland@imj-prg.fr}

\author{Pierre Youssef}
\address {Laboratoire de Probabilit\'es et de Mod\`eles Al\'eatoires\\
Universit\'e Paris-Diderot\\
5 rue Thomas Mann, 75205, Paris CEDEX 13, France}
\email{\tt youssef@math.univ-paris-diderot.fr}

\begin{abstract}
We show that any $n\times m$ matrix $A$ can be approximated in operator norm by a submatrix with a number of columns of order the stable rank of $A$. This improves on existing results by removing an extra logarithmic factor in the size of the extracted matrix. Our proof uses  the recent solution of the Kadison-Singer problem. We also develop a sort of tensorization technique to deal with  constraint approximation problems. As an application, we provide a sparsification result with equal weights and an optimal approximate John's decomposition for non-symmetric convex bodies. This enables us to show that any convex body in $\R^n$ is arbitrary close to another one having $O(n)$ contact points and fills the gap left in the literature after the results of Rudelson and Srivastava by completely answering the problem. As a consequence, we also show that the method developed by Gu\'edon, Gordon and Meyer to establish the isomorphic Dvoretzky theorem yields to the best known result once we inject our improvements.
\end{abstract}

\keywords{Matrix approximation, sparsification, 
Kadison-Singer, 
John's decomposition, isomorphic Dvoretzky}

\subjclass[2010]{Primary 15A60; Secondary 65F50, 52A20, 46B07}

\maketitle

\section{Introduction}

Let $A$ be an $n\times m$ matrix. We denote by $s_i(A)=\sqrt{\lambda_i(A^*A)}$ the singular values of $A$, where $A^*$ denotes the adjoint matrix of $A$, and $\lambda_i (A^*A)$ denotes the $i^{th}$-eigenvalue of $A^*A$ rearranged in non-increasing order.

\vskip 10pt

The singular values measure the isomorphism ``quality'' of $A$ as an operator from $\ell_2^m$ to $\ell_2^n$. Indeed, we have $s_1(A)=\sup_{x\in S^{m-1}} \|Ax\|_2= \|A\|$ the operator norm of $A$, and $s_m(A)=\inf_{x\in S^{m-1}} \|Ax\|_2$. Therefore, one always has 
$$
s_{\min}(A) \|x\|_2 \le \|Ax\|_2 \le s_{\max}(A) \|x\|_2.
$$

When the smallest singular value is non-zero, the operator is injective and the inequalities above assert that $A$ is an isomorphism on its image with distortion the ratio of the largest to the smallest singular value. In a similar way, using the Courant-Fisher formula \cite[Theorem 7.3.8]{HJ}, one can characterize all singular values. Therefore, the sequence of singular values determines the action of the operator $A$.

\vskip 10pt

In this paper we are interested in the following problem. Find a coordinate subspace of $\R^m$ so that the restriction of $A$ to this subspace approximates the action of the operator $A$. This is done by finding $\sigma\subset [m]$ so that $A_\sigma$, the matrix containing the columns indexed by $\sigma$, verifies that 
$$
\|A_\sigma A_\sigma^* -AA^*\|
$$
is small. This insures that the spectrum of $A_\sigma A_\sigma^*$ is close to that of $AA^*$ implying the same for the sequence of singular values.

\vskip 10pt

If one looks for any subspace, not necessarily a coordinate one, then one would choose the subspace spanned by the singular vectors corresponding to the large singular values as the tiny ones are automatically approximated by zero. This suggests that the dimension of the smallest subspace would be the number of big singular values the operator has.

\vskip 10pt

A quantity measuring this is what is often called the {\it stable rank} (or {\it numerical rank}) which is defined as 
$$
\srank(A) \eqdef {\|A\|_{\HS}^2}/{\|A\|^2}
$$
where $\|A\|_{\HS}$ denotes the Hilbert-Schmidt norm of $A$, i.e. $\|A\|_{\HS}=\sqrt{\tr(AA^*)}=\sqrt{\sum_{i \le m} s_i(A)^2}$. It is easy to check that $\srank(A)$ is less or equal than $\rank(A)$. Since $\srank(A)$ is the ratio of the sum of all singular values squared to the squared largest one, it doesn't take into account the tiny singular values. Therefore, it is natural to aim at extracting a number of columns of order $\srank(A)$ so that the corresponding restricted matrix approximates the original one.

\vskip 10pt

We treat this problem while allowing a reweighting of the extracted columns. This corresponds to finding a multi-set $\sigma\subset [m]$ and looking at $A_{\sigma}$. The number of repetition of an index in $\sigma$ gives the value of the weight associated to the corresponding column. We denote by $\widetilde A$ the matrix obtained by normalizing the columns of $A$.

\vskip 10pt

The following theorem is our main result which shows that any $n\times m$ matrix $A$ can be approximated in operator norm by a submatrix with a number of columns of order the stable rank of $A$. It improves on existing results by removing an extra logarithmic factor in the size of the extracted matrix.

\begin{thm} \label{thm-main}
Let $A$ be an $n\times m$ matrix. Then, for any $\e>0$ there exists a multi-set $\sigma \subset [m]$ with $|\sigma| \le \srank(A)/c\e^2$ so that 
$$
\left\|c\e^2\|A\|^2 {\widetilde A_{\sigma}}{\widetilde A_{\sigma}}^*- AA^*\right\|
 \le \e \|A\|^2
$$
where $c>0$ is an absolute constant. Moreover if $A=\kappa \widetilde A$ for some $\kappa>0$, then the above holds with $\sigma$ being a set.
\end{thm}

The conclusion of the above theorem can be also formulated as follows. There exists an $m\times m$ non-negative diagonal matrix with at most $\srank(A)/c\e^2$ non-zero entries so that 
$$
\|ADA^*- AA^*\| \le \e \|A\|^2 .
$$

When $A=\kappa \widetilde A$ for some $\kappa>0$, the non-zero entries of $D$ are given explicitly by $c\e^2\|A\|^2/\kappa^2$.

\vskip 10pt

Results of this kind, which are sometimes called ``column subset selection'' or ``column selection approximation'' attract a lot of attention in the numerical analysis and the analysis of algorithms communities. Normally they address this problem from an algorithmic point of view. Many papers are devoted to this problem, let us mention a few \cites{BMD,BMW,CM,AB, Tro}. We also refer to \cite{tropp-survey} and references therein for a detailed exposition on the topic. 

\smallskip

We are interested in studying this problem from a theoretical point of view. Theorem \ref{thm-main} produces the minimal size approximation compared to the results available in the literature. For instance, the best known bounds on the number of selected columns are of order $\srank(A)\ln( \srank(A))$ (see Table 2 and Theorem 4.1 in \cite{HI}). This is for example done in \cite{RV} using random sampling in which case the logarithmic factor is needed. Thus, we improve by removing the logarithmic factor in the existing results. Theorem \ref{thm-main} is similar in nature to \cite[Corollary 1.2]{Yo1} where one finds an approximation valid only over the range of $A$ which allows to reduce the size of the extraction below the stable rank.

\vskip 10pt

It should be noted that our result is only existential and we do not know how to produce an algorithm achieving the extraction promised in the theorem. The main reason for this is that our proof is based on the solution of the Kadison-Singer problem \cite{MSS} which is not constructive.

\vskip 10pt

Our second interest in this paper is a sort of approximation problem with constraints. One may ask to achieve an approximation of the matrix while keeping some special properties it has. For example, given an $n\times m$ matrix $A$ and a vector $v$ in its kernel, one may be interested in approximating $A$ by a submatrix with fewer columns while keeping $v$ not far from the kernel of the restricted matrix. This is motivated by some geometric applications which will be discussed later. To give a brief idea, looking at the columns of the matrix $A$ as vectors in $\R^n$, we see the vector $v$ from the kernel as some weighted barycenter of the column vectors of $A$. Therefore, the constraint in this case can be seen as keeping $v$ not far from being a weighted barycenter of the selected column vectors. One can also consider constraints of the form $\Vert Av\Vert_2\leq \alpha \Vert v\Vert_2$ for some positive constant $\alpha$ and ask to keep this property almost stable after sparsification. The method we develop allows to achieve this as well as considering multiple constraints. 

\vskip 10pt

Let $A$ be an $n\times m$ matrix. Let $\e>0$ and let $D$ be an $m\times m$ diagonal matrix with non-negative weights on its diagonal. We say that $A$ is $(\alpha,\beta,D)$-approximable matrix if
$$
\alpha AA^*\preceq ADA^*\preceq \beta AA^*. 
$$

We denote this property by $A\in \aprox_{\alpha,\beta}D$. In the case where $\alpha=1-\e$ and $\beta=1+\e$, we simply say $A\in \aprox_{\e}D$. Note that we always have $A\in \aprox_\e I_m$, and if $|\supp(D)|<n$ then $\aprox_\e D$ is empty for any $D$.

\vskip 10pt

Our second main result is the following.

\begin{thm} \label{thm-sparse-with-constraints}
Let $A$ be an $n\times m$ matrix and let $V=(v_i)_{1 \le i \le k}$ be an $m\times k$ matrix. Consider the $(n + k)\times m$ matrix $B$ given by $B^*=(A^*\mid V)$. Let $\e > 0$ and let $D$ be an $m\times m$ diagonal matrix so that $B\in \aprox_\e D$. Then $A, V^*\in \aprox_\e D$, and for any $1 \le i \le k$ 
\begin{align} \label{eq-ADv}
\left\|(AA^*)^{-\frac{1}{2}}A \left[D-(1 + \e)I_n \right]v_i \right\|_2 \le 2\e \|v_i\|_2.
\end{align}
\end{thm}

To better understand the conclusion stated in \eqref{eq-ADv}, suppose that $A$ is isotropic (i.e. $AA^*=I_n$). Then \eqref{eq-ADv} implies that the action of the obtained submatrix $AD^{\frac12}$ on the updated matrix $D^{\frac12}V$ approximates in some sense the one of the original matrix $A$ on $V$. This is checked by controlling the distances between the corresponding columns.

\vskip 10pt

The proof of Theorem~\ref{thm-sparse-with-constraints} uses the simple fact that a coordinate projection of any element of $\aprox_{\e}D$ remains in this set together with some careful matrix analysis.

\vskip 10pt

Theorem~\ref{thm-sparse-with-constraints} should be combined with Theorem~\ref{thm-main} to deal with constraint approximation. The main idea is to implement the constraint into a higher dimensional approximation problem then use the combination of Theorems \ref{thm-main} and \ref{thm-sparse-with-constraints} to deduce the original constraint approximation problem. This idea will be the key behind the applications to the study of contact points of a non-symmetric convex body which will be discussed later in the paper. We believe that Theorem~\ref{thm-sparse-with-constraints} may have other implications for various constraint approximation problems different from the geometric ones discussed here.

\vskip 10pt

The paper is organized as follows. In Section \ref{sec-proof} we prove the two main results of this paper. Then, in Section \ref{sec-new-bss}, we present a corollary about sparsification with equal weights of identity decompositions. Section \ref{sec-john} contains the result concerning approximate John's decompositions which is the key to proving the two applications which are presented in the last two sections. We discuss contact points of convex bodies in Section \ref{sec-contact} and the isomorphic Dvoretzky theorem in Section \ref{sec-dvor}.

\bigskip

\section{Proofs of the main results} \label{sec-proof}

Using the method of interlacing polynomials, Marcus, Spielman and Srivastava \cite{MSS} proved the following result which implies the $KS_2$ conjecture of Weaver \cite{weaver}, which known to be equivalent to the Kadison-Singer problem \cite{KS} (for further refinements of the $KS_2$ conjecture see \cite{casazza}).

\begin{thm} \cite{MSS} \label{thm-mss}
Let $\e > 0$ and let $v_1, \dots, v_m$ be independent random vectors in $\C^n$ with finite support, so that $\E \sum_{i \le m} v_i \otimes v_i =I_n$ and $\E \|v_i\|_2^2 \le \e$ for all $i$. Then
$$
\p \left[ \left\|\sum_{i \le m} v_i \otimes v_i \right\|\le \left(1 + \sqrt{\e}\right)^2 \right] > 0.
$$
\end{thm}

First, let us generalize Theorem \ref{thm-mss} by replacing the isotropic covariance structure of $\sum_{i \le m} v_i \otimes v_i$ by any other one, i.e. random vectors whose sum is not necessarily isotropic. 

\begin{prp} \label{prp-gen-mss}
Let $\delta>0$ and let $v_1, \dots, v_m$ be independent random vectors in $\C^n$ with finite support, so that $B :=\E \sum_{i \le m} v_i \otimes v_i$ is an $n\times n$ positive semi-definite Hermitian matrix and $\E \|v_i\|_2^2 \le \delta$ for all $i$. Then
$$
\p \left[ \sum_{i \le m} v_i \otimes v_i \preceq B + \gamma I_n \right] > 0
$$
where $\gamma:=\gamma(\delta, B)=\|B\|\left[\left(1 + \sqrt{\delta/\|B\|}\right)^2-1 \right]$.
\end{prp}

\begin{proof}
We may assume without loss of generality that $B\preceq I_n$, otherwise we replace $B$ by $B/\|B\|$, $v_i$ by $v_i/\sqrt{\|B\|}$ and $\delta$ by $\delta/\|B\|$.

Denoting $C:=I_n-B$, then $C$ is a positive semi-definite matrix and can be decomposed as $C=\sum_{i \le n} \lambda_i u_i \otimes u_i $, where $\lambda_i \ge 0$ are the non-negative eigenvalues of $C$ and $u_i$ are its unit norm eigenvectors (also of $B$).

Now, for each $i \le n$, if $\lambda_i \le \delta$ then we define $\tilde u_i=\sqrt{\lambda_i} u_i$ so that $\|\tilde u_i\|_2^2 \le \delta$. If $\lambda_i > \delta$, then we proceed by splitting as follows 
$$
\lambda_i u_i \otimes u_i =\delta u_i \otimes u_i + \dots + \delta u_i \otimes u_i + (\lambda_i-\delta \lfloor \frac{\lambda_i}{\delta}\rfloor) u_i \otimes u_i
$$
where there are $\lfloor \frac{\lambda_i}{\delta}\rfloor$ terms equal to $\delta u_i \otimes u_i $ in the previous sum.

Renaming each of the vectors in the sum, we conclude that we can write $C=\sum_{i \le k} \tilde u_i \otimes \tilde u_i$ so that for any $i \le k$ we have $\|\tilde u_i\|_2^2 \le \delta$.

Define the following independent random vectors in $\C^n$:
$$
\tilde v_i=\begin{cases}
v_i & i\le m \\
\tilde u_{i-m} & m + 1 \le i \le m + k. 
\end{cases}
$$

Then the $\tilde v_i$'s satisfy $\E \sum_{i \le m + k} \tilde v_i \otimes \tilde v_i =I_n$ and $\E \|\tilde v_i\|_2^2 \le \delta$. Applying Theorem \ref{thm-mss}, we deduce 
$$
\p \left[ \sum_{i \le m + k} \tilde v_i \otimes \tilde v_i \preceq \left(1 + \sqrt{\delta}\right)^2I_n \right] > 0 .
$$

Noting that $\sum_{i \le m + k} \tilde v_i \otimes \tilde v_i =\sum_{i \le m} v_i \otimes v_i + I_n-B$ completes the from of Proposition \ref{prp-gen-mss}.
\end{proof}

Similarly to what is done in \cite[Corollary 1.5]{MSS}, we deduce the following corollary (which can be generalized to any number of partitions of $[m]$).

\begin{cor} \label{cor-gen-mss}
Let $\delta>0$ and let $A=\sum_{i \le m} u_i \otimes u_i $ be an $n\times n$ matrix where $u_i \in \C^n$ and $\|u_i\|_2^2 \le \delta$. Then there exists $\sigma\subset[m]$ with $|\sigma|\le m/2$ so that 
$$
\left\|2\sum_{i\in\sigma} u_i \otimes u_i -A \right\|\le \gamma
$$
where $\gamma=\gamma(2\delta, A)$ is the constant from Proposition \ref{prp-gen-mss}.
\end{cor}

\begin{proof}
For any $i \le m$, define the following $2n$ dimensional vectors 
$$
w_{i, 1}=\begin{pmatrix}
u_i \\
0_n\\
\end{pmatrix} \quad \text{and}\quad w_{i, 2}=\begin{pmatrix}
0_n\\
u_i \\
\end{pmatrix}
$$
where $0_n$ is the $0$-vector in $\C^n$.

\vskip 10pt

Let $v_1, \dots, v_m \in \C^{2n}$ be independent random vectors so that $v_i=\sqrt2w_{i, j}$ with probability $1/2$ for $j\in\{1, 2\}$. Clearly, we have 
$$
\E\sum_{i \le m} v_i \otimes v_i =
\begin{pmatrix}
A & 0\\
0 & A\\
\end{pmatrix}:=B \quad \text{and} \quad \E\|v_i\|_2^2 \le 2\delta .
$$

Moreover $\|B\|=\|A\|$. Applying Proposition \ref{prp-gen-mss}, we find a realization of the $v_i$'s so that 
$$
\sum_{i \le m} v_i \otimes v_i \preceq B + \gamma I_n
$$
where $\gamma=\gamma(2\delta, A)$. Therefore, there exists $\sigma$ with $|\sigma|\le m/2$ so that 
$$
2 \sum_{i\in\sigma} u_i \otimes u_i \preceq A + \gamma I_n \quad \text{and} \quad 2 \sum_{i\in\sigma^c} u_i \otimes u_i \preceq A + \gamma I_n .
$$

Since $A=\sum_{i\in\sigma} u_i \otimes u_i + \sum_{i\in\sigma^c} u_i \otimes u_i $ we deduce 
$$
A-\gamma I_n\preceq 2 \sum_{i\in\sigma} u_i \otimes u_i \preceq A + \gamma I_n
$$
which finishes the proof.
\end{proof}

We are now ready to prove the first main result of this paper. Its proof uses an iterative procedure based on Corollary \ref{cor-gen-mss}. 

\subsection{Proof of Theorem~\ref{thm-main}}
Denote by $(a_i)_{i \le m}$ the columns of $A$, by $(x_i)_{i \le m}$ its normalized columns, and for any $i \le m$ put $c_i=\|a_i\|_2^2$. 
It is easy to check that 
\begin{align} \label{eq-matrix-B}
B:= AA^*= \sum_{i \le m} c_ix_i \otimes x_i.
\end{align}

By a standard splitting argument, one may assume that all the scalars $c_i$ are equal to $\kappa=\tr(B)/M$ with $M\in\N$ satisfying 
\begin{align} \label{eq-m-big}
M\geq \frac{72\, \tr(B)}{\e^2\|B\|}\quad \text{, i.e.}\quad \sqrt{\frac{2\kappa}{\|B\|}} \le \frac{\e}{6}
\end{align}
so that $B$ can be rewritten as a sum of rank one matrices with equal weights, that is, $B=\kappa \sum_{j=1}^M y_j \otimes y_j$, where $y_j\in\{x_1,\dots,x_m\}$. Note that it's enough to assume having all weights of the same order, this will affect only the constants at the end.

\vskip 10pt

Define $k\in\N$ as the largest integer satisfying 
\begin{align} \label{eq-choice-k}
\frac{M}{2^k} \ge \frac{144}{(\sqrt{2}-1)^2}\frac{\tr(B)}{\e^2 \|B\|}=\frac{144}{\e^2(\sqrt{2}-1)^2}\, \srank(A).
\end{align}

Denote $\sigma_0=[M]$ and $u_i=\sqrt{\kappa}y_i$ so that $B_0:=B=\sum_{i\in \sigma_0} u_i \otimes u_i $ and $\|u_i\|_2^2 \le \kappa$ for any $i\in\sigma_0$.

\vskip 10pt

Now apply Corollary \ref{cor-gen-mss} to find $\sigma_1\subset \sigma_0$ with $|\sigma_1|\le M/2$ so that 
\begin{align} \label{eq-start-iteration}
\left\|2\sum_{i\in\sigma_1} u_i \otimes u_i -B_0 \right\|\le \|B_0\|\left[ \left(1 + \sqrt{2\kappa/\|B_0\|} \right)^2-1 \right].
\end{align}

Denote $B_1=\sum_{i\in \sigma_1} u_i \otimes u_i $ and $\alpha_1=\left(1 + \sqrt{2\kappa/\|B_0\|} \right)^2$. Note that $(2-\alpha_1) \|B_0\|\le \|2B_1\|\le \alpha_1 \|B_0\|$. Applying Corollary \ref{cor-gen-mss}, we find $\sigma_2\subset \sigma_1$ with $|\sigma_2|\le |\sigma_1|/2 \le M/4$ so that 
\begin{align} \label{eq-pre-step2-iteration}
\left\|2\sum_{i\in\sigma_2} u_i \otimes u_i -B_1 \right\|\le \|B_1\|\left[ \left(1 + \sqrt{2\kappa/\|B_1\|}\right)^2-1 \right].
\end{align}

Combining \eqref{eq-start-iteration} and \eqref{eq-pre-step2-iteration} we deduce 
$$
(2-\alpha_1\alpha_2) \|B_0\|\le \|2^2B_2\|\le \alpha_1\alpha_2 \|B_0\|
$$
where $B_2=\sum_{i\in\sigma_2} u_i \otimes u_i $ and $\alpha_2=\left( 1 + \sqrt{2\kappa/\|B_1\|} \right)^2$.

\vskip 10pt

Therefore, we construct by induction 
$\sigma_k\subset\sigma_{k-1}\subset\dots\subset\sigma_0$ satisfying $|\sigma_\ell| \le |\sigma_{\ell-1}|/2$ for any $\ell \le k$ and 
$$
\left\|2^\ell B_\ell-B \right\|\le \|B\|\left[ \prod_{i=1}^\ell \alpha_i-1 \right]
$$
where $B_{\ell}=\sum_{i\in\sigma_\ell} u_i \otimes u_i $ and $\alpha_i=\left(1 + \sqrt{2\kappa/\|B_{i-1}\|} \right)^2$ for any $1 \le i \le \ell$. Moreover 
\begin{align} \label{eq-induction-norms}
(2-\prod_{i=1}^\ell \alpha_i) \|B\|\le \|2^\ell B_\ell\|\le \|B\|\prod_{i=1}^\ell \alpha_i. 
\end{align}

We will show by induction that $\beta_{\ell}:=\prod_{i=1}^\ell\alpha_i \le 1 + \e/2$ for any $\ell \le k$. By \eqref{eq-m-big}, we have that $\alpha_1 \le 1 + \e/2$. Let $\ell <k$ and suppose that $\beta_{i} \le 1 + \e/2$ for any $i \le \ell$. We need to show that $\beta_{\ell + 1} \le 1 + \e/2$. Note that \eqref{eq-induction-norms} together with the induction hypothesis imply that $\|B_i\|\ge 2^{-i} (1-\e/2) \|B\|$ for any $i \le \ell$. Therefore, 
\begin{align*}
\ln \beta_{\ell + 1}&=2\sum_{i=1}^{\ell + 1} \ln \left(1 + \sqrt{2\kappa/\|B_{i-1}\|}\right) L\le 2\sum_{i=1}^{\ell + 1}\sqrt{\frac{2\kappa}{\|B_{i-1}\|}} \\
& \le 2\sum_{i=1}^{\ell + 1}\sqrt{\frac{2^i\kappa}{(1-\frac{\e}{2})\|B\|}}=\frac{2\sqrt{2}}{\sqrt{2}-1}\sqrt{\frac{2^{\ell + 1}\kappa}{(1-\frac{\e}{2})\|B\|}}.
\end{align*}

Since $\ell<k$ then using \eqref{eq-choice-k} and $\e/3 \le \ln(1 + \e/2)$ we deduce that $\ln \beta_{\ell + 1} \le \ln(1 + \e/2)$ and finish the induction. Thus, we found $\sigma_k$ with $|\sigma| \le M/2^k$ so that 
$$
\left\|2^k \sum_{i\in\sigma_k} u_i \otimes u_i -B \right\|\le \frac{\e}{2} \|B\| .
$$

Finally, replacing $u_i$ by $\sqrt{\kappa}y_i$ and using \eqref{eq-choice-k} completes the proof of Theorem \ref{thm-main}. $\square$

\vskip 10pt

Now, we turn to the proof of the second main result which will follow after a careful matrix analysis. We should emphasize again that its value lies in its combination with Theorem~\ref{thm-main} as it will be shown in the next sections. 

\subsection{Proof of Theorem~\ref{thm-sparse-with-constraints}}
The first assertion is trivial since 
$$
P_{n} BB^*P_{n}=AA^* \quad \text{and} \quad P_{n} BDB^*P_{n}=ADA^*
$$
and the same for goes $V^*$.

\vskip 10pt

Application of Schur complement, and the fact that $B, A, V^*\in \aprox_\e D$ yield
\begin{align} \label{eq-rank2-ineq}
K :=\begin{pmatrix}
2\e AA^* & (1 + \e)AV-ADV \\
(1 + \e)V^*A^*-V^*DA^* & 2\e V^*V
\end{pmatrix}
\succeq 0. 
\end{align}

Let $w\in S^{n-1}$ and $\lambda\in\R$. For any $1 \le i \le k$ define the $(n + k)$ dimensional vector $w_i(\lambda)$ given by 
$$
w_i(\lambda)^*=\left(\left((AA^*)^{-\frac{1}{2}}w\right)^*\Big|(\lambda e_i)^* \right), \quad e_i \in \R^k. 
$$

From \eqref{eq-rank2-ineq}, we have $\left\langle Kw_i(\lambda), w_i(\lambda)\right\rangle \ge 0$ for any $1 \le i \le k$ and any $\lambda\in\R$. Therefore, for any $1 \le i \le k$
$$
\e \|v_i\|_2^2 \lambda^2 + \left[(1 + \e) \left\langle (AA^*)^{-\frac{1}{2}}Av_i, w\right\rangle-\left\langle (AA^*)^{-\frac{1}{2}} ADv_i, w\right\rangle \right]\lambda + \e \ge 0. 
$$

Since this should be true for any $\lambda$ then 
$$
\left| \left\langle (AA^*)^{-\frac{1}{2}}ADv_i, w\right\rangle-(1 + \e) \left\langle (AA^*)^{-\frac{1}{2}}Av_i, w \right\rangle \right|\le 2\e \|v_i\|_2.
$$

Since this is true for any $w\in S^{n-1}$ then \eqref{eq-ADv} follows. This completes the proof of Theorem \ref{thm-sparse-with-constraints}.
$\square$

\bigskip

\section{Sparsification with equal weights} \label{sec-new-bss}

A nice consequence of Theorem \ref{thm-main} concerns sparsification results. 
Given an identity decomposition $I_n=\sum_{i \le m} v_i \otimes v_i$, \, $v_1, \dots, v_m\in\R^n$, the goal is to sparsify this decomposition by reducing the number of vectors used while keeping the corresponding quadratic form almost the same. This is addressed in the following corollary of Theorem~\ref{thm-main}. 

\begin{cor} \label{cor-new-bss}
Let $v_1, \dots, v_m\in\R^n$ with $\sum_{i \le m} v_i \otimes v_i=I_n$. Then for any $\e>0$ there exists a multi-set $\sigma \subset [m]$ with $|\sigma| \le {n}/{c\e^2}$ so that 
$$
(1-\e)I_n \preceq \frac{n}{|\sigma|}\sum_{i\in\sigma} \frac{v_i}{\|v_i\|_2} \otimes \frac{v_i}{\|v_i\|_2} \preceq (1+\e)I_n
$$
where $c>0$ is an absolute constant. Moreover, if all the $v_i$'s have the same Euclidean norm, then the above holds with $\sigma$ being a set.
\end{cor}

The conclusion of this corollary can be reformulated as follows: there exists a sequence of non-negative integers $(\kappa_i)_{i\leq m}$ with $\sum_{i\leq m} \kappa_i \leq {n}/{c\e^2}$ such that 
$$
(1-\e)I_n \preceq \frac{n}{\sum_{i\leq m}\kappa_i} \sum_{i\leq m} \kappa_i \frac{v_i}{\|v_i\|_2} \otimes \frac{v_i}{\|v_i\|_2} \preceq (1+\e)I_n. 
$$
Moreover, when the $v_i$'s have the same Euclidean norm, then all positive weights are equal to $\frac{n}{m\, \|v_i\|_2^2}$, i.e. $\kappa_i= 1$ for any $i$ such that $\kappa_i\neq 0$. 

\vskip 10pt

We should note that a similar sparsification result, with no information on the weights, was previously obtained by Batson, Spielman and Srivastava \cite{BSS} by a different method (see also \cite[Appendix~F]{HO} where a similar statement appears implicitly). 
Thanks to Theorem~\ref{thm-main}, the above corollary will follow easily. 

\begin{proof}[Proof of Corollary~\ref{cor-new-bss}] 
The proof is based on applying Theorem \ref{thm-main} with the $n\times m$ matrix $A=(v_1,\dots,v_m)$, and use that $AA^*=I_n$.

\vskip 10pt

More precisely, following the previous proof, the iterative process, with $B=AA^*=I_n$ (as defined in \eqref{eq-matrix-B}), one finds $\sigma_k$ with $|\sigma| \le M/2^k$ so that 
\begin{equation}\label{eq-rem-cor}
\left(1-\frac{\e}{2}\right)I_n\preceq 2^k\sum_{i\in\sigma_k} u_i \otimes u_i \preceq \left(1+\frac{\e}{2}\right)I_n.
\end{equation}

Taking the trace on all sides and using that $\|u_i\|_2^2= n/M$, we deduce 
$$
\left(1-\frac{\e}{2}\right) \le 2^k\frac{|\sigma_k|}{M} \le \left(1+\frac{\e}{2}\right) .
$$

Therefore, 
$$
2^k \sum_{i\in\sigma_k} u_i \otimes u_i\preceq \frac{n}{|\sigma_k|} \sum_{i\in\sigma_k} y_i \otimes y_i \preceq \frac{2^k}{1-\e/2} \sum_{i\in\sigma_k} u_i \otimes u_i .
$$

This, together with \eqref{eq-rem-cor}, shows that 
$$
(1-\e/2)I_n\preceq \frac{n}{|\sigma_k|} \sum_{i\in\sigma_k} y_i \otimes y_i\preceq \frac{1+\e/2}{1-\e/2} I_n
$$
which concludes the proof of Corollary \ref{cor-new-bss}.
\end{proof}

\bigskip

\section{Approximate John's decompositions} \label{sec-john}

Given an arbitrary convex body $K$ in $\R^n$, John's theorem \cite{J} states that among the ellipsoids contained in $K$, there exists a unique ellipsoid of maximal volume. This ellipsoid is called the John's ellipsoid of $K$. If the John's ellipsoid of $K$ happens to be $B_2^n$, the body $K$ is said to be in John's position. For any $K$ there is a linear invertible transformation $T:\R^n\to \R^n$ so that $TK$ is in John's position. If $K$ is in John's position then there are $m\le n(n+3)/2$ contact points $x_1, \dots, x_m\in \partial K \cap \partial B_2^n$ and positive weights $c_1, \dots, c_m$ so that
\begin{equation} \label{eq-john}
\sum_{i=1}^m c_i x_i\otimes x_i=I_n , \quad \sum_{i=1}^m c_i x_i=0. 
\end{equation}

We call this collection $\{x_i, c_i\}_{i\le m}$ a John's decomposition of the identity (of the body $K$).

\vskip 10pt

An important problem is to approximate a John's decomposition by extracting vectors from $\{x_i,c_i\}$ (of course, as less as possible) so that their decomposition is still close to the identity (clearly the corresponding weights have to be adapted as well) and their disposition is still close to being balanced. Note that in the symmetric case, the second condition in \eqref{eq-john} trivially holds, as if $x$ is a contact point then $-x$ is as well. Thus, in the symmetric case, Corollary~\ref{cor-new-bss} gives a possible solution to this problem as the balancing condition can be obtained automatically. However, in the non-symmetric case, the condition $\sum_{i=1}^m c_i x_i = 0$ is meaningful and we regard it, in this context, as a constraint.

\vskip 10pt

The following theorem addresses exactly this situation which can be seen as a constraint approximation problem. 

\begin{thm} \label{thm-equal-john}
Let $\{x_i, c_i\}_{i\le m}$ be a John's decomposition of the identity. Then for any $\e>0$ there exists a multi-set $\sigma \subset [m]$ with $|\sigma| \le {n}/{c\e^2}$ so that 
$$
(1-\e) I_n \preceq \frac{n}{|\sigma|} \sum_{i\in\sigma} (x_i-u)\otimes (x_i-u)\preceq (1 + \e) I_n
$$
where $u=\frac{1}{|\sigma|} \sum_{i\in\sigma} x_i$ satisfies $\|u\|_2\le \frac{2\e}{3\sqrt{n}}$, and $c>0$ is an absolute constant.
\end{thm}

It may be useful to note that the conclusion of the above theorem can be formulated as follows. There exists a sequence of non-negative integers $\{\kappa_i\}_{i \le m}$ with $\sum_{i \le m} \kappa_i \le {n}/{c\e^2}$ so that 
$$
(1-\e) I_n \preceq \frac{n}{\sum_{i \le m} \kappa_i} \sum_{i \le m} \kappa_i (x_i-u)\otimes (x_i-u)\preceq (1 + \e) I_n
$$
where $u=\frac{ \sum_{i \le m}\kappa_i x_i}{\sum_{i \le m} \kappa_i}$ satisfies $\|u\|_2\le \frac{2\e}{3\sqrt{n}}$, and $c>0$ is an absolute constant. Moreover, if all the $c_i$'s are equal then all the non-zero $\kappa_i$'s are equal to $1$.

\vskip 10pt

We should note that this improves on Srivastava's theorem \cite[Theorem 5]{Sr} in three ways. First we obtain an approximation whose ratio $(1+\e)/(1-\e)$ can be made arbitrary close to $1$ while in Srivastava's result one could only get a $(4+\e)$-approximation. The second improvement concerns the dependence on $\e$ in the estimate of the norm of $u$: Srivastava obtains a similar bound with $\e$ replaced by $\sqrt{\e}$. Finally, Theorem \ref{thm-equal-john} gives an explicit expression of the weights appearing in the approximation.

\vskip 10pt

In Sections \ref{sec-contact}-\ref{sec-dvor} we present two applications of geometric flavor of Theorem \ref{thm-equal-john}. In these applications the fact that there is a control on the magnitude of the weights in the approximate John's decompositions is crucial.

\vskip 10pt

Before we proceed with the proof of Theorem \ref{thm-equal-john}, we need the following corollary of Theorem~\ref{thm-sparse-with-constraints}. 

\begin{cor} \label{cor-kernel}
Let $A$ be an $n\times m$ matrix and let $v\in \ker A$. Consider the $(n + 1)\times m$ matrix $B$ given by $B^*=(A^*\mid v)$. Let $\e > 0$ and let $D$ be an $m\times m$ diagonal matrix so that $B\in \aprox_\e D$. Then $A, v\in \aprox_\e D$ and 
$$
\left\|(AA^*)^{-\frac{1}{2}}ADv \right\|_2 \le 2\e \|v\|_2.
$$
Moreover, if $C=AD^{\frac{1}{2}}-\frac{1}{\|D^{\frac{1}{2}}v\|_2^2}ADv\otimes D^{\frac{1}{2}} v$ then $D^{\frac{1}{2}}v\in \ker C$ and 
$$
\left(1-\e-\frac{4\e^2}{1-\e} \right) AA^*\preceq CC^*\preceq (1 + \e) AA^*. 
$$
\end{cor}

\begin{proof}
Let $A$ be an $n\times m$ matrix and $v\in \ker A$. Consider the $(n + 1)\times m$ matrix $B$ given by $B^*=(A^*\mid v)$. Let $\e > 0$ and let $D$ be an $m\times m$ diagonal matrix so that $B\in \aprox_\e D$.

By Theorem \ref{thm-sparse-with-constraints}, we have $A, v\in \aprox_\e D$. By the definition of $C$, we have $D^{\frac{1}{2}}v\in \ker C$. An easy calculation shows that 
$$
CC^*=ADA^*-\frac{1}{\left\|D^{\frac{1}{2}}v\right\|_2^2}ADv\otimes ADv. 
$$

This implies that $CC^*\preceq ADA^*\preceq (1 + \e)AA^*$. Since $v\in \ker A$ then by \eqref{eq-ADv}, we have
$$
\left\|(AA^*)^{-\frac{1}{2}}ADv \right\|_2 \le 2\e \|v\|_2.
$$

This implies that $(AA^*)^{-\frac{1}{2}}ADv\otimes ADv (AA^*)^{-\frac{1}{2}} \preceq 4\e^2 \|v\|_2^2 I_n$ which means that 
\begin{align} \label{eq-ADv-tensor}
ADv\otimes ADv \preceq 4\e^2 \|v\|_2^2 AA^*.
\end{align}

Since $v\in \aprox_\e D$ then $\|v\|_2^2\le \frac{\|D^{\frac{1}{2}} v\|_2^2}{1-\e}$. This together with \eqref{eq-ADv-tensor} gives 
$$
\frac{1}{\|D^{\frac{1}{2}}v\|_2^2}ADv\otimes ADv\preceq \frac{4\e^2}{1-\e} AA^*.
$$

This together with the fact that $ADA^*\succeq (1-\e) AA^*$ finishes the proof of Theorem \ref{cor-kernel}.
\end{proof}

Combining the previous statement with Corollary~\ref{cor-new-bss}, we will be able to prove Theorem~ \ref{thm-equal-john}.

\begin{proof}[Proof of theorem \ref{thm-equal-john}]
Let $\e>0$. Let $A=(\sqrt{c_i}x_i)$ be an $n\times m$ matrix and let $v$ be the $m$ dimensional vector with $\sqrt{c_i/n}$ as coordinates. Clearly, we have $AA^*=I_n$, and by the assumption $\sum_{i\le m} c_i x_i=0$ we have $v \in \ker A$.

Consider the $(n + 1)\times m$ matrix $B=(b_1,\dots,b_m)$ given by $B^*=(A^*\mid v)$, and note that $BB^*=I_{n + 1}$. This means that the columns of $B$, obtained by concatenating the original vectors of the decomposition together with the corresponding weights, form an identity decomposition\footnote{This fact was used by Ball in \cite{Ba} where he proves a reverse isoperimetric inequality.} in $\R^{n+1}$.

Applying Corollary \ref{cor-new-bss} with $b_i^*=(\sqrt{c_i}x_i^*|\sqrt{c_i/n})$ for any $i \le m$ and $\e/3$ instead of $\e$, we find a multi-set $\sigma \subset [m]$ with $|\sigma| \le {n}/{c\e^2}$ so that 
$$
\frac{n}{|\sigma|}\sum_{i\in\sigma} \frac{b_i \otimes b_i}{c_i} \simeq_{\frac{\e}{3}} I_{n+1} .
$$

This means that $B\in \aprox_{\frac{\e}{3}} D$, where $D$ is the diagonal matrix with entries defined as follows 
$$
d_{ii} = \frac{\kappa_i n}{c_i|\sigma|} 
$$
where $\kappa_i$ is the number of appearances of the index $i$ in $\sigma$. Recall that $\sigma$ is a mutliset of indices from $[m]$, thus $\sum_{i\le m} \kappa_i=|\sigma|$. Note that 
$$
ADv = \frac{\sqrt n}{|\sigma|} \sum_{i\le m} \kappa_i x_i = \frac{\sqrt n}{|\sigma|} \sum_{i\in\sigma} x_i=\sqrt n u 
$$
where we denote $u=\frac{1}{|\sigma|} \sum_{i\in\sigma} x_i$. We denote by $D^{1/2}$ the square root of $D$, i.e. the diagonal matrix with diagonal entries 
$\delta_{ii} = \sqrt{d_{ii}} = \frac{\sqrt{\kappa_i n}}{\sqrt{c_i |\sigma|}}$.

In the same way, we get
$$
D^{\frac{1}{2}}v = (\sqrt{\kappa_i/|\sigma|})_{i\le m} \quad \text{and} \quad AD^{\frac{1}{2}} = \sqrt{\frac{n}{|\sigma|}}(\sqrt{\kappa_i} x_i)_{i\le m}
$$

Thus, the matrix $C:=AD^{\frac{1}{2}}-\frac{1}{\|D^{\frac{1}{2}}v\|_2^2}ADv\otimes D^{\frac{1}{2}} v$ has its columns equal to 
$$
\sqrt{\frac{n}{|\sigma|}} \sqrt{\kappa_i}(x_i-u). 
$$

Applying Corollary \ref{cor-kernel} yields the result. Finally, the fact that $\|ADv\|_2\le \frac{2\e}{3} \|v\|_2$ implies that $\|u\|_2\le \frac{2\e}{3\sqrt{n}}$.
\end{proof}

\bigskip

\section{Contact points of convex bodies} \label{sec-contact}

The points of intersection of the body $K$ and its ellipsoid of maximal volume, i.e. the contact points, play a crucial role in the understanding of the geometry of convex bodies.

\vskip 10pt

An important problem is to control the number of contact points of a convex body $K \subset \R^n$ in John's position. If it is centrally symmetric then the number of contact points satisfying \eqref{eq-john} is $n\le m \le n(n+1)/2$. This result was proved, for example, in \cite[Section 16]{T-J}. In fact, it is optimal, as studied by Pe{\l}czy{\'n}ski and Tomczak-Jaegermann \cite{PT-J}. In particular, they proved that for any $n\le m\le n(n+1)/2$ there exists a centrally symmetric convex body whose John's ellipsoid has exactly $m$ contact points. If $K$ is not centrally symmetric then more contact points might be needed and the estimate is $n(n+3)/2=M$, as showed by Gruber \cite{Grub}. He also proved that the set of convex bodies having less than $M$ contact points is of the first Baire category in $\mathcal{K}$, the set of all convex bodies in $\R^n$.

\vskip 10pt

Let us recall the definition of the Banach-Mazur distance between two convex bodies $K$ and $H$ in $\R^n$:
$$
d(K,H)=\underset{T\in GL(n), u\in\R^n}{\inf}\Big\{\alpha: H+u\subseteq TK\subseteq \alpha (H+u)\Big\}.
$$

Given $K\in \mathcal{K}$, Rudelson \cite{R2} showed that there exists a body $H$ arbitrary close to $K$, that is $d(H,K) \le 1+\e$, with a much smaller number of contact points, at most $Cn\ln n/\e^2$. Later, Srivastava \cite{Sr} was able to remove the logarithmic factor in the above theorem at the expense of finding the body $H$ at distance $\sqrt{5}+\e$ instead of $1+\e$. In the case where the body $K$ is symmetric, Srivastava shows the existence of a convex body $H$ having at most $32n/\e^2$ contact points so that $d(H,K) \le 1+\e$.

\vskip 10pt

The gap between these two results \cites{R2,Sr} remained open till know. The following theorem presents a unified solution which fills this gap.

\begin{thm} \label{thm-contact-points}
Let $K$ be a convex body in $\R^n$. Then for any $\e>0$ there exists a convex body $H\subset \R^n$ so that $H$ has at most $C n/\e^2$ contact points with its John's ellipsoid and 
$$
H\subset K\subset (1 + \e) H
$$
where $C>0$ is an absolute constant.
\end{thm}

Let us only sketch the proof as it follows literally from Rudelson's approach \cite{R2} once we inject the results proved earlier in this paper. The first step in Rudelon's proof is to approximate the John's decomposition of the identity given by the body $K$. As we already mentioned, a loss of a logarithmic factors appears at this place in Rudelson's proof as he provides an approximate John's decomposition with $c(\e)n\log n$ vectors. 
This approximate John's decomposition is then used in the second step where he constructs the approximating body $H$. Thus, Theorem \ref{thm-contact-points} follows by substituting (as a black box) Lemma 3.1 of \cite{R2} (step 1 of Rudelson's construction \cite[Section 4]{R2}) by Theorem \ref{thm-equal-john}.

\bigskip

\section{Isomorphic Dvoretzky} \label{sec-dvor}

Another interesting application is an isomorphic version of Dvoretzky's theorem.

\vskip 10pt

Given a convex body $K\subset \R^n$ and $1\le k \le n$, the isomorphic version of Dvoretzky's theorem asks for an upper bound on the minimal distance of a $k$-dimensional section of $K$ to $B_2^k$, the euclidean ball of dimension $k$.

\vskip 10pt

This problem has been extensively studied in the literature (e.g. see \cites{MS1,MS2,GGM,LT-J}). Although this problem is settled, we are interested in the method developed in \cite{GGM} which had an extra logarithmic factor. Our contribution consists of showing that this method, which uses Rudelson's theorem alongside some inequalities of Gaussian processes established by Gordon \cite{Go1,Go2}, provides the optimal isomorphic Dvoretzky's theorem once combined with Theorems \ref{thm-equal-john} and \ref{thm-contact-points}.

\vskip 10pt

The following theorem appears as Theorem 6.7 in \cite{LT-J}.

\begin{thm} \label{thm-iso-dvor}
Let $K$ be a convex body in $\R^n$. Then for any $k \le cn$, there exists a $k$-dimensional affine subspace $F$ of $\R^n$ so that 
$$
d(F\cap K, B_2^k) \le C\sqrt{\frac{k}{\ln(1+\frac{n}{k})}}
$$
where $C,c>0$ are universal constants.
\end{thm}

We are able to provide a proof of this theorem as a simple consequence of the argument introduced by Gordon, Gu\'edon and Meyer \cite{GGM}, who obtained the same statement with an extra logarithmic factor in the dimension of the subspace. Their proof consists of two main steps. Firstly, they use Rudelson's result (see \cite{R2,R3}) to obtain a new body which is close enough to $K$ but has few contact points with its John's ellipsoid and the John's decomposition of the identity has all the weights of the same order of magnitude\footnote{This is why Srivastava's result \cite[Theorem 5]{Sr} couldn't be applied here.}. Secondly, they reduce the study to a polytope which corresponds to the contact points of the new body. In this part, they use some inequalities of Gordon \cite{Go1,Go2}, and a variant of the Dvoretzky-Rogers lemma \cite{DR}.

Thus, verbally repeating their proof and replacing Rudelson's result by Theorems~\ref{thm-equal-john} and~\ref{thm-contact-points} yields the above Theorem~\ref{thm-iso-dvor} with no extra $\log$ factor compared to what is shown in \cite{GGM}.


\bigskip

\begin{thebibliography}{10}

\bib{AB}{article}{
 author={Avron, Haim},
 author={Boutsidis, Christos},
 title={Faster subset selection for matrices and applications},
 journal={SIAM J. Matrix Anal. Appl.},
 volume={34},
 date={2013},
 number={4},
 pages={1464--1499},
}

\bib{Ba}{article}{
 author={Ball, K.},
 title={Volume ratios and a reverse isoperimetric inequality},
 journal={J. London Math. Soc. (2)},
 volume={44},
 date={1991},
 number={2},
 pages={351--359},
}

\bib{BMW}{article}{
 author={Barg, Alexander},
 author={Mazumdar, Arya},
 author={Wang, Rongrong},
 title={Restricted isometry property of random subdictionaries},
 journal={IEEE Trans. Inform. Theory},
 volume={61},
 date={2015},
 number={8},
 pages={4440--4450},
}

\bib{BSS}{article}{
 author={Batson, J.}, 
 author={Spielman, D.A.}, 
 author={Srivastava, N.}, 
 title={Twice-Ramanujan sparsifiers}, 
 journal={SIAM J. Comput.}, 
 volume={41}, 
 date={2012}, 
 number={6}, 
 pages={1704--1721}, 
}

\bib{BMD}{article}{
 author={Boutsidis, Christos},
 author={Mahoney, Michael W.},
 author={Drineas, Petros},
 title={An improved approximation algorithm for the column subset
 selection problem},
 conference={
 title={Proceedings of the Twentieth Annual ACM-SIAM Symposium on
 Discrete Algorithms},
 },
 book={
 publisher={SIAM, Philadelphia, PA},
 },
 date={2009},
 pages={968--977},
}

\bib{casazza}{article}{
 author={Bownik, Marcin},
 author={Casazza, Peter G.},
 author={Marcus, Adam W.},
 author={Speegle, Darrin},
 title={Improved bounds in weaver and feichtinger conjectures}, 
 journal={http://arxiv.org/abs/1508.07353},
}

\bib{CM}{article}{
 author={{\c{C}}ivril, A.},
 author={Magdon-Ismail, M.},
 title={Column subset selection via sparse approximation of SVD},
 journal={Theoret. Comput. Sci.},
 volume={421},
 date={2012},
 pages={1--14},
}

\bib{DR}{article}{
 author={Dvoretzky, A.},
 author={Rogers, C.A.},
 title={Absolute and unconditional convergence in normed linear spaces},
 journal={Proc. Nat. Acad. Sci. U. S. A.},
 volume={36},
 date={1950},
 pages={192--197},
}

\bib{Go1}{article}{
 author={Gordon, Y.},
 title={Some inequalities for Gaussian processes and applications},
 journal={Israel J. Math.},
 volume={50},
 date={1985},
 number={4},
 pages={265--289},
}

\bib{Go2}{article}{
 author={Gordon, Y.},
 title={Majorization of Gaussian processes and geometric applications},
 journal={Probab. Theory Related Fields},
 volume={91},
 date={1992},
 number={2},
 pages={251--267},
}

\bib{GGM}{article}{
 author={Gordon, Y.}, 
 author={Gu{\'e}don, O.}, 
 author={Meyer, M.}, 
 title={An isomorphic Dvoretzky's theorem for convex bodies}, 
 journal={Studia Math.}, 
 volume={127}, 
 date={1998}, 
 number={2}, 
 pages={191--200}, 
}

\bib{Grub}{article}{
 author={Gruber, P.M.},
 title={Minimal ellipsoids and their duals},
 journal={Rend. Circ. Mat. Palermo (2)},
 volume={37},
 date={1988},
 number={1},
 pages={35--64},
}

\bib{tropp-survey}{article}{
 author={Halko, N.},
 author={Martinsson, P.G.},
 author={Tropp, J.A.},
 title={Finding structure with randomness: probabilistic algorithms for
 constructing approximate matrix decompositions},
 journal={SIAM Rev.},
 volume={53},
 date={2011},
 number={2},
 pages={217--288},
}

\bib{HO}{article}{
 author={Harvey, N.J.A.},
 author={Olver, N.},
 title={Pipage rounding, pessimistic estimators and matrix concentration},
 conference={
 title={Proceedings of the Twenty-Fifth Annual ACM-SIAM Symposium on
 Discrete Algorithms},
 },
 book={
 publisher={ACM, New York},
 },
 date={2014},
 pages={926--945},
}

\bib{HI}{article}{
 author={Holodnak, J.T.}, 
 author={Ipsen, I.C.F.}, 
 title={Randomized approximation of the Gram matrix: exact computation and
 probabilistic bounds}, 
 journal={SIAM J. Matrix Anal. Appl.}, 
 volume={36}, 
 date={2015}, 
 number={1}, 
 pages={110--137}, 
}

\bib{HJ}{book}{
 author={Horn, R.A.},
 author={Johnson, C.R.},
 title={Matrix analysis},
 edition={2},
 publisher={Cambridge University Press, Cambridge},
 date={2013},
 pages={xviii+643},
}

\bib{J}{article}{
 author={John, F.}, 
 title={Extremum problems with inequalities as subsidiary conditions}, 
 conference={
 title={Studies and Essays Presented to R. Courant on his 60th
 Birthday, January 8, 1948}, 
 }, 
 book={
 publisher={Interscience Publishers, Inc., New York, N. Y.}, 
 }, 
 date={1948}, 
 pages={187--204}, 
}


\bib{KS}{article}{
author={Kadison, R.},
author={Singer, I.},
title={Extensions of pure states},
journal={Amer. J. Math},
volume={81},
date={1959},
pages={383--400},
}

\bib{LT-J}{article}{
 author={Litvak, A.E.},
 author={Tomczak-Jaegermann, N.},
 title={Random aspects of high-dimensional convex bodies},
 conference={
 title={Geometric aspects of functional analysis},
 },
 book={
 series={Lecture Notes in Math.},
 volume={1745},
 publisher={Springer, Berlin},
 },
 date={2000},
 pages={169--190},
}

\bib{MSS}{article}{
 author={Marcus, A.W.}, 
 author={Spielman, D.A.}, 
 author={Srivastava, N.}, 
 title={Interlacing families II: Mixed characteristic polynomials and the
 Kadison-Singer problem}, 
 journal={Ann. of Math. (2)}, 
 volume={182}, 
 date={2015}, 
 number={1}, 
 pages={327--350}, 
}

\bib{MS1}{article}{
 author={Milman, Vitali},
 author={Schechtman, Gideon},
 title={An ``isomorphic'' version of Dvoretzky's theorem. II},
 conference={
 title={Convex geometric analysis},
 address={Berkeley, CA},
 date={1996},
 },
 book={
 series={Math. Sci. Res. Inst. Publ.},
 volume={34},
 publisher={Cambridge Univ. Press, Cambridge},
 },
 date={1999},
 pages={159--164},
}

\bib{MS2}{article}{
 author={Milman, Vitali D.},
 author={Schechtman, Gideon},
 title={An ``isomorphic'' version of Dvoretzky's theorem},
 language={English, with English and French summaries},
 journal={C. R. Acad. Sci. Paris S\'er. I Math.},
 volume={321},
 date={1995},
 number={5},
 pages={541--544},
}

\bib{PT-J}{article}{
 author={Pe{\l}czy{\'n}ski, A.},
 author={Tomczak-Jaegermann, Nicole},
 title={On the length of faithful nuclear representations of finite rank
 operators},
 journal={Mathematika},
 volume={35},
 date={1988},
 number={1},
 pages={126--143},
}



\bib{R2}{article}{
 author={Rudelson, M.}, 
 title={Contact points of convex bodies}, 
 journal={Israel J. Math.}, 
 volume={101}, 
 date={1997}, 
 pages={93--124}, 
}

\bib{R3}{article}{
 author={Rudelson, M.}, 
 title={Random vectors in the isotropic position}, 
 journal={J. Funct. Anal.}, 
 volume={164}, 
 date={1999}, 
 number={1}, 
 pages={60--72}, 
}



\bib{RV}{article}{
 author={Rudelson, M.}, 
 author={Vershynin, R.}
 title={Sampling from large matrices: an approach through geometric functional analysis}, 
 journal={J. ACM}, 
 volume={54}, 
 date={2007},
 number={4}, 
 pages={Art. 21, 19 pp}, 
}



\bib{Sr}{article}{
 author={Srivastava, N.}, 
 title={On contact points of convex bodies}, 
 conference={
 title={Geometric aspects of functional analysis}, 
 }, 
 book={
 series={Lecture Notes in Math.}, 
 volume={2050}, 
 publisher={Springer, Heidelberg}, 
 }, 
 date={2012}, 
 pages={393--412}, 
}

\bib{T-J}{book}{
 author={Tomczak-Jaegermann, Nicole},
 title={Banach-Mazur distances and finite-dimensional operator ideals},
 series={Pitman Monographs and Surveys in Pure and Applied Mathematics},
 volume={38},
 publisher={Longman Scientific \& Technical, Harlow; copublished in the
 United States with John Wiley \& Sons, Inc., New York},
 date={1989},
 pages={xii+395},
}

\bib{Tro}{article}{
 author={Tropp, Joel A.},
 title={Column subset selection, matrix factorization, and eigenvalue
 optimization},
 conference={
 title={Proceedings of the Twentieth Annual ACM-SIAM Symposium on
 Discrete Algorithms},
 },
 book={
 publisher={SIAM, Philadelphia, PA},
 },
 date={2009},
 pages={978--986},
}

\bib{weaver}{article}{
author={Weaver, N.},
title={The Kadison-Singer problem in discrepancy theory},
journal={Discrete Math.},
date={2004},
number={278},
pages={no. 1--3, 227--239},
}

\bib{Yo1}{article}{
 author={Youssef, P.}, 
 title={A note on column subset selection}, 
 journal={Int. Math. Res. Not. IMRN}, 
 date={2014}, 
 number={23}, 
 pages={6431--6447}, 
}

\end{thebibliography}
\end{document}